\documentclass[12pt,letterpaper]{article}
\usepackage{amssymb, amsfonts}
\usepackage{amsmath}
\usepackage{amscd}
\usepackage[all]{xy}
\usepackage{amsthm}
\usepackage{enumerate}   
\usepackage{xcolor}
\usepackage{tikz-cd}

\usepackage{url}           

\newtheorem{thm}{Theorem} 
\newtheorem{prop}{Proposition}
\newtheorem{lem}{Lemma}
\newtheorem{cor}{Corollary}

\theoremstyle{definition}
\newtheorem{definition}{Definition}
\newtheorem{expl}{Example}
\newtheorem{conj}{Conjecture}

\newtheorem{question}{Question}
\newtheorem{rem}{Remark}

\DeclareMathOperator{\R}{\mathbb{R}}
\DeclareMathOperator{\C}{\mathbb{C}}
\DeclareMathOperator{\Z}{\mathbb{Z}}
\DeclareMathOperator{\N}{\mathbb{N}}

\DeclareMathOperator{\Q}{\mathbb{Q}}

\DeclareMathOperator{\Gal}{\text{Gal}}
\DeclareMathOperator{\Hom}{\text{Hom}}

\DeclareMathOperator{\Ext}{\text{Ext}}

\DeclareMathOperator{\prim}{\text{prim}}
\DeclareMathOperator{\Fil}{\text{Fil}}
\DeclareMathOperator{\rank}{\text{rank}}

\DeclareMathOperator{\gr}{\text{gr}}

\DeclareMathOperator{\W}{\mathcal{W}}
\DeclareMathOperator{\fin}{\text{fin}}
\DeclareMathOperator{\cpt}{\text{cpt}}
\DeclareMathOperator{\tor}{\text{tor}}
\DeclareMathOperator{\free}{\text{free}}

\DeclareMathOperator{\Tam}{\text{Tam}}
    \DeclareFontFamily{U}{wncy}{}
    \DeclareFontShape{U}{wncy}{m}{n}{<->wncyr10}{}
    \DeclareSymbolFont{mcy}{U}{wncy}{m}{n}
    \DeclareMathSymbol{\Sha}{\mathord}{mcy}{"58}

\numberwithin{equation}{section}

\title{Heights and Tamagawa numbers of motives}
\author{Tung T. Nguyen}

\begin{document}
\maketitle


\begin{abstract}
K. Kato has recently defined and studied heights of mixed motives and proposed some interesting questions. In this article, we relate the study of heights to the study of Tamagawa numbers of motives. We also partially answer one of Kato's questions about the number of mixed motives of bounded heights in the case of mixed Tate motives with two graded quotients. Finally, we provide a concrete computation with the number of mixed Tate motives with three graded quotients.
\end{abstract}

\section{Introduction and main results}
K. Kato has recently defined and studied heights of mixed motives and proposed some interesting questions (see $\cite{[Kato1]}$). In this article, we relate the study of heights to the Tamagawa number conjecture for motives. More precisely, we have the following theorem. 
\begin{thm} \label{theorem1}
Let $M$ be a pure motives with integer coefficients of weight $-d$ such that $d \geq 3$. We assume further that $M$ has semistable reduction at all places. Then
\[ \lim_{B \to \infty} \dfrac{\# \{x \in B(\Q)| H_{\diamond,d}(x) \leq B \}}{\mu \left(x \in \prod'_{p \leq \infty} B(\Q_p)| H_{\diamond,d}(x) \leq B \right)} =\dfrac{1}{\Tam(M)} .\]
Here $B(\Q)$ (respectively $B(\Q_p)$) is the relevant motivic cohomology associated with $M$ (respectively the local Selmer group at $p$), $\mu$ is the Tamagawa measure associated with $M$, and $\Tam(M)$ is the Tamagawa number of $M$ as  defined in Bloch-Kato's article \cite{[BK]}. Finally, $H_{\diamond, d}(x)$ is the height function of mixed motives defined in Kato's article \cite{[Kato1]}. 
\end{thm}

In Theorem $1$, we restrict ourselves to the case of pure motives because the Tamagawa number conjecture was formulated this way in $\cite{[BK]}$. We expect that a similar statement happens if $M$ is a mixed motive. To demonstrate this point, we will give a concrete computation (up to a power of $2$) of the terms appeared in the left hand side of Theorem $\ref{theorem1}$ when $M$ is a mixed Tate motive with graded quotients $\Z(m), \Z(n)$. More precisely, we have the following.
\begin{thm} \label{thm:theorem_2}
Let $m,n$ be two natural number such that $m-n \geq 2$, $m$ is even and $n$ is odd. Let $D$ be a mixed Tate motive with graded quotients $\Z(m)$ and $\Z(n)$. Then 
\[ \# \{x \in B(\Q)| H_{\star, \diamond} (x) \leq B \} \sim 2^{t} \dfrac{\Sha(D)}{(n-1)! \zeta(n) \zeta(1-m)} \log(B)^n . \] 

Here $\Sha(D)$ is the Tate-Shafarevich group associated to $D$ as defined in \cite{[BK]} and \cite{[FP]}. 
\end{thm} 
This theorem answers one of Kato's questions $\cite{[Kato1]}$ about the number of mixed motives of bounded heights.

Bloch-Kato's approach is to study $B(\Q)$, which should be considered the extension group $\Ext^1(\Z, M)$ of a mixed motive, and how it is distributed in its corresponding extension groups of the realizations. It is natural to ask whether it is possible to generalize this study to allow more fixed graded quotients. To show that this question is interesting and important, we will provide a concrete computation with the set of motives with graded quotients $\Z(12), \Z(3), \Z$. More precisely, we have the following theorem. 
\begin{thm} \label{theorem3}
Let $X$ be the set of all mixed motives with graded quotients $\Z(12), \Z(3), \Z$. Then, we have 

\[ \# \{x \in X| H_{\star, \diamond} (x) \leq B \} \sim \frac{1}{ 8! 2! {12 \choose 3}} \dfrac{\Sha(3)}{\zeta(3)}  \dfrac{\Sha(9)}{\zeta(9)} \dfrac{\Sha(12)}{\zeta(-11)} \left(\dfrac{2}{691}-\dfrac{1}{691^2} \right) \log(B)^{12} .\]

Here for each integer $i$,  $\Sha(i)$ is the Tate-Shafarevich group associated with the Tate motive $\Z(i)$. 
\end{thm} 
In the above example, there are no contributions to heights from the non-archimedean places. In general, this will not be the case. As K. Kato suggested, in order to understand contributions from archimedean places (respectively non-archimedean places), we need to study period domains (respectively $p$-adic period domains) classifying Hodge structures (respectively $p$-adic Hodge structures). This study has been carried out in $\cite{[Kato2]}$.

\section{Tamagawa numbers of motives} \label{Tamagawa}
In this section, we briefly review the definition of the Tamagawa measure associated to a motive with $\Z$ coefficients. We note that there is still no universally agreed definition of mixed motives even up to this day. For our purpose, we will follow Jannsen's definition: that is a mixed motive with $\Q$ coefficients over a number field $F$ is a collection of realizations with weight and Hodge filtrations, and these realizations are related by comparison isomorphisms which are compatible with both these filtrations (for more details, see $\cite{[Jannsen2]}$). For simplicity, we will assume that $F=\Q$ as Bloch-Kato did in their paper. Also, to define heights, we will assume that each graded quotient $\gr_{\omega}^{W}(M)$ (which is a pure motive of weight $\omega$) is equipped with a polarization, $\gr_{\omega}^{W}(M)=0$ for $w \geq -2$, and $M$ has semistable reduction everywhere. In particular, each graded quotient $\gr_{\omega}^{W}(M)$ is a polarized pure motive with semistable reduction. We note that the condition $\gr_{\omega}^{W}(M)=0$ for all $\omega \geq -2$ guarantees that the Tamagawa measure is well-defined (the product of local Tamagawa measure converges). In principle, our argument should work even without this restriction.  

To define Tamagawa measures and heights, we will restrict ourselves to motives with $\Z$-coefficients: that is a pair $(M, \Theta)$ where $M$ is a mixed motive with $\Q$ coefficients, and $\Theta$ is a free $\Z$-module of finite rank equipped with a linear action of $G_{\Q}$ on $\widehat{\Theta}=\widehat{\Z} \otimes \Theta$ and an isomorphism $\Q \otimes \Theta \cong M_{B}$ where $M_{B}$ is the Betti realization of $M$.

\begin{rem}
This definition of mixed motives with $\Z$ coefficients is almost the same as the one defined in section $11$ of Fontaine's paper $\cite{[Fo2]}$. 
\end{rem}

Following Bloch-Kato, we define 
\[ A(\Q_p) = \begin{cases} H_{f}^1(\Q_p, \widehat{\Theta}) &\mbox{if } p<\infty \\
(D_{\infty} \otimes_{\R} \C)/(\Fil^{0} D_{\infty} \otimes_{\R} \C)+\Theta)^{+}& \mbox{if } p=\infty \end{cases}.  \]
We define $B(\R)=A(\R)$ and $B(\Q_p)$ to be the inverse image in $H^1_{g}(\Q_p, \widehat{\Theta})$ of 
\[ \text{Im} \left(\Psi \to H^1_{g}(\Q_p, \Theta \otimes \textbf{A}_{\Q}^{f})/H^1_{f}(\Q_p, \Theta \otimes \textbf{A}_{\Q}^{f}) \right) .\]
Here $\textbf{A}_{\Q}^{f}$ is the ring of finite adele of $\Q$ and $\Psi$ is a certain motivic cohomology group associated with $M$. More precisely, in the case $M=H^m(X)(r)$ with $X$ is a proper smooth scheme over $\Q$, Bloch-Kato defined $\Psi$ as (see \cite[Section 5]{[BK]}) 

\[ \Psi=\begin{cases} \text{gr}^{r}(K_{2r-m-1}(X) \otimes \Q) &\mbox{if } m \neq 2r-1 \\ 
(CH^{r}(X) \otimes \Q)_{\text{homologically $\sim$ 0}}&\mbox{if } m=2r-1. 
\end{cases} \]

We equip $B(\Q_p)$ with the unique topology such that the topology of $B(\Q_p)/A(\Q_p)$ is discrete. Roughly speaking, these are elements of $H^1_{g}(\Q_p, \widehat{\Theta})$ such that the monodromy operators have motivic origin (for more details on this see discussions in section $\ref{local heights}$ as well as Section $2.1.6$ of $\cite{[Kato1]}$). Finally, we define $A(\Q) \subset B(\Q)$ using $\Psi$ instead of $\Phi$ where $\Phi$ is some integral version of $\Psi$. In the case $M=H^m(X)(r)$ as above, Bloch-Kato defined $\Phi$ as follow (see \cite[Section 5]{[BK]} for more details). Suppose that $\mathfrak{X}$ is a proper regular model of $X$ and $m \neq 2r-1, 2r-2$. Then $\Phi$ is the image of $\text{gr}^r(K_{2r-m-1}(\mathfrak{X}) \otimes \Q)$ in $\Psi$. The Tamagawa number of $(M, \Theta)$ is defined to be (see also \cite[Equation 5.11]{[BK]})
\[ \Tam(M)=\mu \left(\prod_{p} A(\Q_p)/A(\Q) \right) .\] 
The Tamagawa number conjecture of Bloch-Kato says that 
\begin{conj} 
\[ \Tam(M)= \frac{\# H^0(\Q, \Theta^{*} \otimes \Q/\Z(1))}{\# \Sha(M)},\]
where $\Sha(M)$ is the Tate-Shafarevich group associated with $M$ as defined in \cite[Equation 5.13]{[BK]}.
\end{conj} 
 
Note that the above sets all depend on the choice of $\Theta$, so we should write $B_{\Theta}(\Q)$ instead of $B(\Q)$. However, for simplicity, we will drop $\Theta$ in all related notations.

\section{Heights of mixed motives}
Let $(M, \Theta)$ be a motive with $\Z$-coefficients whose graded quotients are polarized. In this section, we review the some key facts about the height functions $H_{\diamond,p,d}$ on local spaces $B(\Q_p)$, the height function $H_{\diamond,d}$ on its associated adelic space, and the height function on the global space $B(\Q)$. 

\label{conjectures}
The definition of height functions depends on the validity of some conjectures in arithmetic geometry. We refer to Section $1.7.2$ of $\cite{[Kato1]}$ for precise statements about these conjectures. Throughout this chapter, we will assume that these conjectures hold. 

\subsection{Height function on $B(\Q_p)$} \label{local height}
Let $\widehat{M}$ be a mixed motive with $\Q$ coefficients whose graded quotients are polarized. As mentioned earlier, we will also assume that $\widehat{M}$ has semistable reductions. Let $v$ be a place of $\Q$ associated to a prime number $p$ (we will often identify them) and $\widehat{M}_{\ell}$ be a semistable $\ell$-adic representation of $\Gal(\overline{\Q_p}/\Q_p)$. Define $A=\Q_{\ell}$, $\widehat{\Upsilon}=\widehat{M}_{\ell}$ if $\ell \neq p$, and $A=\Q_{p}^{\text{ur}}$, $  \widehat{\Upsilon}=D_{st}(\widehat{M}_{\ell})$ if $\ell=p$. Since $\widehat{M}_{\ell}$ is a semistable representation of $\Gal(\overline{\Q_p}/ \Q_p)$, there is a nilpotent monodromy operator $N: \widehat{\Upsilon} \to \widehat{\Upsilon}(-1)$ preserving the induced weight filtration on $\widehat{\Upsilon}$. In what follows, we will assume that $N$ induces an increasing filtration $\W_{m}$ (the relative monodromy filtration) on $\widehat{\Upsilon}$ with the following properties: 
\begin{enumerate} \label{WMC}
\item $N \W_{m} \subset \W_{m-2}$ for all $m \in \Z$. 
\item For every $\omega \in \Z$ and $m \geq 0$, $N^m$ induces an isomorphism 
\[ N^{m}: \gr_{\omega +m}^{\W} \gr_{\omega}^{W}(\widehat{\Upsilon}) \cong \gr_{\omega-m}^{\W} \gr_{\omega}^{W} (\widehat{\Upsilon}) .\]
\item The weight monodromy conjecture holds for $\widehat{\Upsilon}$ with respect to this filtration.
\end{enumerate}
We refer to Section $1.7.2$ of $\cite{[Kato1]}$ for a precise statement of the third condition. For simplicity, we call these properties $WMC$. By some linear algebra arguments, we can show that if $\widehat{M}$ is an extension of two motives which verify $WMC$ and the class $[\widehat{M}]$ satisfies some local conditions (see example $\ref{concrete example}$ for a more thorough discussion of this) then $\widehat{M}$ also has semistable reductions and verifies $WMC$. In general, because $\widehat{M}$ is a successive extension of its graded quotients, it is enough to check that its graded pieces' realizations satisfy $WMC$. For pure motives of the form $H^n(X)(r)$ where $X$ is either a curve, surface, abelian variety, or a complete intersection in projective spaces, the $WMC$ holds by works of many mathematicians (see $\cite{[WMC]}$ for a survey of known results on this conjecture.)



As explained in $\cite{[Kato1]}$, for each $d \geq 2$, we have a canonical element
\begin{equation} \label{monodromy}
N_{v,d} \in \left( \gr_{-2}^{\mathcal{W}} \gr_{-d}^{W} \Hom_{A}(\widehat{\Upsilon}, \widehat{\Upsilon}) \right)_{\prim} .
\end{equation} 
We also have a non-degenerate $A$-linear pairing
\[ \langle , \rangle_{N_v}: \left( \gr_{-2}^{\mathcal{W}} \gr_{-d}^{W} \Hom_{A}(\widehat{\Upsilon}, \widehat{\Upsilon}) \right)_{\prim} \times \left( \gr_{-2}^{\mathcal{W}} \gr_{-d}^{W} \Hom_{A}(\widehat{\Upsilon}, \widehat{\Upsilon}) \right)_{\prim} \to A ,\] 
defined as 
\[ \langle a ,b \rangle_{N_v}=\langle \text{Ad}(N_v)^{d-2}(a), b \rangle_{N_v} .\]
Finally, define $l_{v}=|\langle N_{v,d} , N_{v,d} \rangle_{N_v}|^{1/d}$ and
\begin{equation} \label{height function}
H_{\diamond,d,v}(M_{\ell}):=p^{l_{v}}. 
\end{equation}
We expect that the following is true. 
\begin{conj} \label{pairing}
If $M_\ell$ is the $\ell$-adic realization of a mixed motive $M$ then $\ell_{v}$ does not depend on $\ell$. Moreover, $\langle N_{v,d} , N_{v,d} \rangle_{N_v}\in \Q_{\geq 0}$. It is zero if and only if $N_{v,d}=0$. 
\end{conj} 
Suppose $M=H^n(X)(r)$ where $X$ is a smooth projective variety over $\Q$. We assume further that $X$ has proper strictly semistable reduction at $p$. In this case, using the weight spectral sequence of Rapoport-Zink (for $\ell \neq p)$ or the weight spectral sequence of Mokrane (for $\ell =p)$, we have a rather concrete understanding of the above monodromy operator $N$. If $\dim(X) \leq 2$, then we can show that conjecture $\ref{pairing}$ holds.

Let us give a concrete example of the above discussion.
\begin{expl} \label{concrete example}
Let $\widehat{M}$ be a mixed motive which is a successive extension of a pure motive $M$ of weight $-d$ and the Tate motive $\Q$. By taking the $\ell$-adic realization, we have a short exact sequence of $\Gal(\overline{\Q_p}/\Q_p)$-modules 
\begin{equation} \label{extension}
0 \to M_\ell \to \widehat{M}_\ell \to \Q_\ell \to 0 .
\end{equation}
Define $\widehat{\Upsilon}$ and $\Upsilon$ to be the $A$-module associated with $\widehat{M}$ and $M$ as explained above. Then, $\widehat{\Upsilon}$ has the induced filtration with graded quotients $\gr_{-d}^{W} \widehat{\Upsilon}=\Upsilon$ and $\gr_{0}^{W} \widehat{\Upsilon}=A$. We have
\[ \gr_{-d}^{W} \Hom(\widehat{\Upsilon}, \widehat{\Upsilon})=\oplus_{w} \Hom (\gr_{w}^{W} \widehat{\Upsilon}, \gr_{w-d}^{W} \widehat{\Upsilon}).\]
In our case, only $\omega=0$ gives a nontrivial summand. Therefore, we have
\[ \gr_{-d}^{W} \Hom(\widehat{\Upsilon}, \widehat{\Upsilon})=\Hom(A, \Upsilon)=\Upsilon .\] 
This isomorphism is compatible with the relative filtration on both sides, so we have 
\[ \left(\gr_{i}^{\W} \gr_{-d}^{W} \Hom(\widehat{\Upsilon}, \widehat{\Upsilon}) \right) =\left( \gr_{i}^{\W} \Upsilon \right), \forall i \in \Z .\]
Moreover, under this isomorphism $\text{Ad}({\widehat{N}_{v}})$ corresponds to $N$ and we have a commutative diagram
\[\xymatrix@C+1pc{
      \gr_{-2}^{\W} \gr_{-d}^{W} \Hom(\widehat{\Upsilon}, \widehat{\Upsilon}) \ar[d]^{\text{Ad}({\widehat{N}_{v}})^{d-2}} \ar[r] & gr_{-2}^{\W} \Upsilon \ar[d]^{N^{d-2}}    \\
      \gr_{-2d+2}^{\W} \gr_{-d}^{W} \Hom(\widehat{\Upsilon}, \widehat{\Upsilon})  \ar[r] & gr_{-2d+2}^{\W} \Upsilon \\
 }\] 
With this identification, we can interpret the pairing 
\[ \left(\gr_{-2}^{\W} \gr_{-d}^{W} \Hom(\widehat{\Upsilon}, \widehat{\Upsilon}) \right)_{\prim} \times \left( \gr_{-2}^{\W} \gr_{-d}^{W} \Hom(\widehat{\Upsilon}, \widehat{\Upsilon}) \right)_{\prim} \to A ,\]
as the pairing 
\[ \left(gr_{-2}^{\W} \Upsilon \right)_{\prim} \times \left(gr_{-2}^{\W}  \Upsilon) \right)_{\prim} \to A .\]
Concretely, this is described by 
\[ \langle{v, w} \rangle= \langle{N^{d-2}(v), w} \rangle, \forall v, w \in gr_{-2}^{\W} \Upsilon .\] 
More generally, if $M$ has a weight filtration $W_{\omega}$ then the pairing 
\[ \langle , \rangle_{N_v}: \left( \gr_{-2}^{\mathcal{W}} \gr_{-d}^{W} \Hom_{A}(\Upsilon, \Upsilon) \right)_{\prim} \times \left( \gr_{-2}^{\mathcal{W}} \gr_{-d}^{W} \Hom_{A}(\Upsilon, \Upsilon) \right)_{\prim} \to A ,\] 
is the sum of the induced pairing on the graded quotients; i.e 
\[ \langle , \rangle_{N_v}=\sum_{\omega \in \Z}\langle , \rangle_{\omega, N_v}. \]
Here $\langle , \rangle_{\omega, N_v}$ is the pairing associated with $\gr_{\omega}^{W} M$ explained above.

For each $\ell$ adic relization, there is class $[\widehat{M}_{\ell}] \in H^1(\Q_p, M_{\ell})$ associated with the short exact sequence of $\Gal(\overline{\Q_p}/\Q_p)$ modules 
\[ 0 \to M_{\ell} \to \widehat{M}_{\ell} \to \Q_{\ell} \to 0. \] 
Suppose $M$ has semistable reductions everywhere. Then, $\widehat{M}$ has semi-stable reduction everywhere if and only $[\widehat{M}_{\ell}] \in H^1_{g}(\Q_p, M_{\ell})$ for all $\ell$. Here we use the convention that 
\[ H^1_{g}(\Q_p,M_{\ell}) = \begin{cases} H^1(\Q_p, M_{\ell}) &\mbox{if } \ell \neq p \\
\ker (H^1(\Q_p, M_p) \to H^1(\Q_p, M_p \otimes B_{dr})) & \mbox{if } \ell=p. \end{cases}  \]
\end{expl}
Next, we explain how to define the height function 
\[ H_{v,\diamond,d, \ell}: B(\Q_p) \to \R_{\geq 1} .\] 
By definition, there is a canonical map 
\[ B(\Q_p) \to H_{g}^1(\Q_p, \Theta \otimes \textbf{A}_{\Q}^{f})/H^1_{f}(\Q_p, \Theta \otimes \textbf{A}_{\Q}^{f}):=\prod_{\ell} H^1_{g}(\Q_p, M_{\ell})/H^1_{f}(\Q_p, M_{\ell}).\] 

Under the assumption that $M$ has semistable reductions, for $\ell \neq p$, there is an isomorphism 
\[ H^1_{g}(\Q_p, M_{\ell})/H^1_{f}(\Q_p, M_\ell) \cong (M_{\ell}(-1)/NM_{\ell})^{\varphi=1} .\]
Here $N$ is the monodromy operator associated with $M_{\ell}$. Furthermore, if $M_{\ell}$ satisfies the weight monodromy conjecture, then we have a canonical isomorphism 
\[ (M_{\ell}(-1)/N M_{\ell})^{\varphi=1}=\left( (\gr_{-2}^{\W} M_{\ell}(-1))_{\prim} \right)^{\varphi=1} .\]
Let $a_{\ell} \in H^1_{g}(\Q_p, M_{\ell})$. Then $a_{\ell}$ corresponds to an extension of $\Gal(\overline{\Q_p}/\Q_p)$-modules: 
\[ 0 \to M_{\ell} \to ({E_a})_{\ell} \to \Q_\ell \to 0 .\]
The image of $a_{\ell}$ under the isomorphism 
\[ H^1_{g}(\Q_p, M_{\ell})/H^1_{f}(\Q_p, M_\ell) \cong \left( (\gr_{-2}^{\W} M_{\ell}(-1))_{\prim} \right)^{\varphi=1} ,\] 
is nothing but the monodromy operator $N_{a_{\ell},v,d}$ of $(E_a)_{\ell}$, see $\ref{monodromy}$. Here we use the identification (see example $\ref{concrete example}$)
\[ \left(\gr_{-2}^{\W} \gr_{-d}^{W} \Hom((E_a)_{\ell}, (E_a)_{\ell} \right)_{\prim} =\left( \gr_{-2}^{\W} M_{\ell} \right)_{\prim}.\]
For $a \in B(\Q_p)$ let $a_{\ell}$ be the image of $a$ in $\in H_{g}^1(\Q_p, M_{\ell})/H_{f}^1(\Q_p, M_{\ell})$. As before, we define 
\[ l_{v,d}(a_{\ell})=\langle N_{a_{\ell}, v,d} , N_{a_{\ell}, v,d} \rangle^{1/d}.\] 

Similarly, when $\ell=p$ there is a canonical isomorphism
\[ H_{g}^1(\Q_p, M_{p})/H^1_{f}(\Q_p, M_{p}) \cong (\gr_{-2}^{\W} D_{st}(M_p))_{\prim}(-1))^{\varphi=1} .\]   
Let $a_p \in H^1_{g}(\Q_p, M_p)$ be the image of $a \in B(\Q_p)$ in $ H^1_{g}(\Q_p, M_p)$. Then $a_p$ corresponds to an extension of $\Gal(\overline{\Q_p}/\Q_p)$ modules 
\[ 0 \to M_{p} \to ({E_a})_{p} \to \Q_p \to 0.\]
Moreover, under the isomorphism 
\[H_{g}^1(\Q_p, M_{p})/H^1_{f}(\Q_p, M_{p}) \cong (\gr_{-2}^{\W} D_{st}(M_p))_{\prim}(-1))^{\varphi=1},\] 
$a_p$ is mapped to the monodromy operator $N_{a_p, v,d}$ of $(E_a)_p$. Here we again use the identification
\[ \left(\gr_{-2}^{\W} \gr_{-d}^{W} \Hom(D_{st}((E_a)_p), D_{st}((E_a)_p)) \right)_{\prim} =\left( \gr_{-2}^{\W} D_{st}(M_p) \right)_{\prim}.\]
Define 
\[ l_{v}(a_{p})=\langle N_{a_{p}, v,d} , N_{a_{p}, v,d} \rangle^{1/d}.\] 
Conjecture $\ref{pairing}$ then implies that for $a \in B(\Q_p)$, $l_{v}(a_{\ell})$ is a non-negative real number and is independent of $\ell$. Moreover, it is $0$ if and only if $a_{\ell}$ belongs to $H^1_{f}(\Q_p, M_{\ell})$ for some $\ell$ (hence for all $\ell$). From now on, we will write $l_{v}(a)$ for this value without referring to the choice of $\ell$. 
\begin{rem}
With the above interpretation, it would be reasonable to think about $H^1_{g}(\Q_p, \Theta \otimes \textbf{A}_{\Q}^{f})/H^1_{f}(\Q_p, \Theta \otimes \textbf{A}_{\Q}^{f})$ as the ``space of monodromy operators" at the place $v$. Similarly, we can think about $B(\Q_p)$ as the space of monodromy operators which have motivic origins.

\end{rem}

We have the following lemma.
\begin{lem} \label{Hg/Hf}
Under the above assumptions, $H_{g}^1(\Q_p, \widehat{\Theta})/H^1_{f}(\Q_p, \widehat{\Theta})$ is a free $\widehat{\Z}$-module of finite rank. 
\end{lem} 
\begin{proof}
Let $\widehat{\Theta}_{\ell}$ be the $\ell$ component of $\widehat{\Theta}$. To prove this lemma, it is enough to show that $H^1_{g}(\Q_p, \widehat{\Theta}_{\ell})/H_{f}^1(\Q_p, \widehat{\Theta}_{\ell})$ is a free $\Z_{\ell}$ module of finite rank and moreover, this rank does not depend on $\ell$. First, we prove that $H^1_{g}(\Q_p, \widehat{\Theta}_{\ell})/H_{f}^1(\Q_p, \widehat{\Theta}_{\ell})$ is a torsion free $\Z_{\ell}$-module. Suppose $a \in H^1_{g}(\Q_p, \widehat{\Theta}_{\ell})$ and $m \in \Z$, $m \neq 0$ such that $ma \in H^1_{f}(\Q_p, \widehat{\Theta}_{\ell})$. By definition, the image of $ma$ in $H^1(\Q_p, M_{\ell})$ belongs to $H^1_{f}(\Q_p, M_{\ell})$. Because $H^1(\Q_p, M_{\ell})$ is a vector space over $\Q_{\ell}$ and $m \neq 0$, the image of $a$ must belong to $H^1_{f}(\Q_p, M_{\ell})$ as well. Hence, $a \in H^1_{f}(\Q_p, \widehat{\Theta}_{\ell})$. This shows that $H^1_{g}(\Q_p, \widehat{\Theta}_{\ell})/H_{f}^1(\Q_p, \widehat{\Theta}_{\ell})$ is torsion free.

We have 
\[ \dfrac{H^1_{g}(\Q_p, \widehat{\Theta}_{\ell})}{H^1_{f}(\Q_p, \widehat{\Theta}_{\ell})} \otimes_{\Z_{\ell}} \Q_{\ell} \cong \dfrac{H^1_{g}(\Q_p, M_{\ell})}{H^1_{f}(\Q_p, M_{\ell})} .\]
By previous discussions, we have the following isomorphisms 

\[ H^1_{g}(\Q_p,M_{\ell})/H^1_{f}(\Q_p, M_{\ell}) = \begin{cases} \left( (\gr_{-2}^{\W} M_{\ell}(-1))_{\prim} \right)^{\varphi=1}  &\mbox{if } \ell \neq p \\ 
\left((\gr_{-2}^{\W} D_{st}(M_p))_{\prim}(-1) \right)^{\varphi=1} & \mbox{if } \ell=p. \end{cases}  \]
As a corollary of Conjecture $1.7.2$ in $\cite{[Kato1]}$, which we assume throughout this article, we have
\[ \dim_{\Q_{\ell}}\left( (\gr_{-2}^{\W} M_{\ell}(-1))_{\prim} \right)^{\varphi=1} =\dim_{\Q_p}\left((\gr_{-2}^{\W} D_{st}(M_p))_{\prim}(-1) \right)^{\varphi=1} .\]
Hence, $\dim_{\Q_{\ell}} \left(H^1_{g}(\Q_p,M_{\ell})/H^1_{f}(\Q_p, M_{\ell}) \right)$ is independent of $\ell$. We conclude that $$H^1_{g}(\Q_p, \widehat{\Theta}_{\ell})/H_{f}^1(\Q_p, \widehat{\Theta}_{\ell})$$ is a free $\Z_{\ell}$ module of finite rank which is independent of $\ell$. 
\end{proof}
We have the following inclusions
\[ A(\Q_p):= H_{f}^1(\Q_p, \widehat{\Theta}) \subset B(\Q_p) \subset H_{g}^1(\Q_p, \widehat{\Theta}).\] 
By definition $B(\Q_p)/A(\Q_p)$ is the inverse image in  $H_{g}^1(\Q_p, \widehat{\Theta})/H^1_{f}(\Q_p, \widehat{\Theta})$ of
\[ \text{Im} (\Psi \to H_{g}^1(\Q_p, \widehat{\Theta} \otimes \textbf{A}_{\Q}^{f})/H^1_{f}(\Q_p, \widehat{\Theta}\otimes \textbf{A}_{\Q}^{f})) .\]
By Lemma $\ref{Hg/Hf}$, we can see that $B(\Q_p)/A(\Q_p)$ is a free $\Z$-module of finite rank. Hence, as a topological group $B(\Q_p)$ is isomorphic to the direct product of $A(\Q_p) \times \Z^{s_p}$ for some $s_p \in \Z_{\geq 0}$. The measure on $B(\Q_p)$ is the tensor product of the measure on $A(\Q_p)$ defined in $\cite{[BK]}$ and the counting measure on $\Z^{s_p}$.

It is natural to ask the following question. 
\begin{question} \label{rank1}
Suppose $M=H^n(X)(r)$ where $X$ smooth projective variety over $\Q$.  Let $\mathfrak{X}$ be a proper regular model of $X$ over $\Z$. Can we describe $s_p$ from the geometry of the special fiber $\mathfrak{X}_{p}$ of $\mathfrak{X}$ at $p$? 

\end{question} 
Under the above assumption, we now can define a height function 
\begin{equation} \label{local heights}
H_{\diamond, v,d}: B(\Q_p) \to \mathbb{R}_{\geq 0}. 
\end{equation}
For $a \in B(\Q_p)$
\[ H_{\diamond,p,d}=N(v)^{l_{v}(a)}=p^{l_{v}(a)} .\] 
\begin{definition}
\noindent
\begin{enumerate}
\item We call $H_{\diamond,p,d}: B(\Q_p) \to \R_{\geq 1}$ the height function on $B(\Q_p)$.
\item Let $h_{\diamond, p,d}: B(\Q_p) \to \mathbb{R}_{\geq 0}$ be defined by 
 \[ h_{\diamond, p,d}:=(\ln H_{\diamond,p,d}) .\] 
We call $h_{\diamond,p,d}$ the logarithmic height function on $B(\Q_p)$. 
\end{enumerate}
\end{definition} 
By definition, $H_{\diamond,p,d}$ and $h_{\diamond,p,d}$ are trivial when restricting to $A(\Q_p)$. Therefore, they are well-defined on the quotient $B(\Q_p)/A(\Q_p)$ which is isomorphic to  $\Z^{s_p}$ for some integer $s_p$. We use the same notations $H_{\diamond,p,d}$ and $h_{\diamond,p,d}$ to denote the induced height function and logarithmic height function on $B(\Q_p)/A(\Q_p)$. Once we fix a basis for $\Z^{s_p}$, $h_{\diamond, p,d}^d$ the $d-$ power of the logarithmic height function  is given by a positive definite quadratic form $Q_{\diamond, p,d}$ on $\Z^{s_p}$. A direct consequence of this observation is the following.
\begin{prop} \label{finite}
For each positive real number $B$ 
\[ \mu(a \in B(\Q_p)| H_{\diamond, p,d}(a) \leq  B) <\infty .\] 
\end{prop} 
\begin{proof}
Let $a \in B(\Q_p)$, we denote by $(a_0,n)$ the image of $a$ under the identification $B(\Q_p)= A(\Q_p) \times \Z^{s_p}$. Then 
\begin{align*}
\{a \in B(\Q_p)| H_{\diamond, p,d}(a)< B \} &= A(\Q_p) \times \{ n \in \Z^{r}| H_{\diamond, p,d}(n) <B \} \\
                                            &=A(\Q_p) \times \{n \in \Z^{r}| h_{\diamond,p,d}(n)< \ln(B) \} \\                                            
                                            &=A(\Q_p) \times \{n \in \Z^r| Q_{\diamond, p,d}(n) < \ln(B)^{d}\}.                                                                        
\end{align*}
Because $Q_{\diamond,p,d}$ is a positive definite quadratic form, there are only finitely many $n \in \Z^{r}$ such that $Q_{\diamond, p,d}(n) < \log(B)^{d}$. Therefore 
\[ \mu(a \in B(\Q_p)| H_{\diamond, p,d}(a)< B)=|\{n \in \Z^r| Q_{\diamond, p,d}(n) < \log(B)^{d}\}| \mu(A(\Q_p))<\infty .\]

\end{proof}
We end this subsection with the following general remark. 
\begin{rem} \label{interesting height functions}
If $M$ is a motive of the form $H^n(X)(r)$ where $X$ is a smooth projective variety over $\Q$, then by the proof of Weil's conjecture we can easily see that $H_{g}^1(\Q_p,\Theta \otimes \textbf{A}_{\Q}^{f})/H_{f}^1(\Q_p, \Theta \otimes \textbf{A}_{\Q}^{f})=0$ if $X$ has good reduction at $p$. Consequently, we have $B(\Q_p)=A(\Q_p)$ and the height function on $B(\Q_p)$ is trivial. Hence, the height functions $H_{\diamond,v,d}$ are only interesting at places $v$ where $X$ has bad reductions.

When $X$ is a curve or a surface, by using the Rapoport-Zink spectral sequence for $\ell$-adic cohomology (when $\ell \neq p$) and the Mokrane spectral sequence for $p$-adic cohomology, we have a concrete understanding of the height functions $H_{\diamond, v,d}$ from the geometry of the special fibers at $p$. It is expected that such an understanding also exists in higher dimensions.

\end{rem}

\subsection{Height function on $B(\R)$}
We define a height function $H_{\diamond, \infty, d}$ on $B(\R)$ as an analogue of the height functions $H_{\diamond, v,d}$ on $B(\Q_p)$ discussed in the previous section. Let $v=\infty$ be the only archimedean place of $\Q$ corresponding to the canonical embedding $\Q \hookrightarrow \R$.

We give some preparations. Let $\widehat{H}$ be a mixed $\R$-Hodge structure with a weight filtration $W$ and a Hodge filtration $F$. As explained in Section $1.7.6$ of $\cite{[Kato1]}$, there exists a unique pair $(s,\delta)$ such that
\begin{enumerate}
\item  $s$ is a splitting $s: \gr^{W} \widehat{H}_{\R} \cong \widehat{H}_{\R}$.
\item $\delta$ is a nilpotent linear map $\delta: \gr^{W} \widehat{H}_{\R} \to \gr^{W} \widehat{H}_{\R}$.
\item The filtration $F$ of $\widehat{H}$ is given by $F=s(\exp(i\delta) \gr^{W}F)$. Moreover, the Hodge $(p,q)$ components $\delta_{p,q}$ of $\delta$ is $0$ unless $p<0$ and $q<0$.
\end{enumerate}
For $d \geq 2$, Let 
\[ \widehat{H}'=\bigoplus_{\omega \in \Z} ((\gr_{\omega}^{W} \widehat{H})^{*} \otimes \gr_{\omega-d}^{W} \widehat{H})=\gr_{-d}^{W} \Hom(\widehat{H}, \widehat{H}) .\] 
Note that $\widehat{H}'$ is polarized Hodge structure of weight $-d$. Let $N_{v, d} \in \widehat{H}'_{\R}$ be the weight $(-d)$-component of the nilpotent linear map $\delta_{v}$ associated with $\widehat{H}_{\R}$.  By definiton, 
\[ N_{v,d} \in \bigoplus_{\omega \in \Z} \Hom(\gr_{\omega}^{W} \widehat{H}_{\R}, \gr_{\omega-d}^{W} \widehat{H}_{\R})=\widehat{H}' .\] 
Because $\widehat{H}'$ is a polarized Hodge structure of weight $-d$, it is possible to define
\[ l(\infty,d)=l_{v,d}(\widehat{H})=\langle N_{v,d}, N_{v,d} \rangle^{1/d} .\] 
\begin{rem}
In his paper $\cite{[Kato1]}$, K. Kato uses the notation $\delta_{v,d}$ for $N_{v,d}$.  We change his notation because we want to emphasize the analogies between height functions at archimedean and non-archimedean places. 

\end{rem}

We will give a concrete example to demonstrate the above discussion. This example is an analogue of example $\ref{concrete example}$.
\begin{expl} \label{concrete example 2}
Let $H$ be a $\R$-Hodge structure of pure weight $-d<0$. Let $\widehat{H}$ be a mixed $\R$-Hodge structure with graded quotients $\gr_{-d}^{W} \widehat{H}=H$ and $\gr_{0}^{W} \widehat{H}=\C$. By definition, $N_{v,d'}=0$ for all $d' \neq d$ and 
\[ N_{v,d} \in  \bigoplus_{\omega \in \Z} \Hom(\gr_{\omega}^{W} \widehat{H}_{\R}, \gr_{\omega-d}^{W} \widehat{H}_{\R})=\Hom(\R, H_{\R})=H_{\R}.\]

There is an exact sequence in the category of mixed $\R$-Hodge structure which is a natural analogue of the exact sequence $\ref{extension}$:
\[ 0 \to H \to \widehat{H} \to \R \to 0 \] 
in the category $\R \text{MHS}$ of mixed $\R$-Hodge structure. By definition, $[\widehat{H}]$ defines a class $[\widehat{H}]$ in $\text{Ext}_{\R \text{MHS}}^1(\R, H)$. We have a canonical isomorphism
\[ \text{Ext}^1_{\R \text{MHS}}(\R, H) \cong H_{\C}/(H_{\R}+F^{0} H_{\C}) .\] 
Moreover, if we define 
\[ H_{\R}^{\leq -1, \leq -1}:=H_{\R} \cap \bigoplus_{p, q \leq -1, p+q=-d} H_{\C}^{p,q}, \]
then we also have a canonical isomorphism 
\[ H_{\C}/(H_{\R}+F^{0} H_{\C}) \cong H_{\R}^{\leq -1, \leq -1} .\]
Therefore, we have an isomorphism 
\[ \text{Ext}^1_{\R \text{MHS}}(\R, H) \cong H_{\R}^{\leq -1, \leq -1} .\]
Under this isomorphism, the class $[\widehat{H}]$ is mapped to the $N_{v,d}$ of $\widehat{H}$. 
\end{expl}
Consequently, we can define a height function $H_{\diamond,\infty,d}: \text{Ext}^1_{\R \text{MHS}}(\R, H) \to \R_{\geq 1}$ as follow. For each $a \in \text{Ext}^1_{\R \text{MHS}}(\R, H)$, let $E_{a}$ be the corresponding mixed $\R$-Hodge structure. We define
\[ H_{\diamond, \infty,d}(a)=\exp(l_{\infty,d}(E_a)) .\] 
\begin{definition}
\noindent
\begin{enumerate}
\item We call $H_{\diamond,\infty,d}: \text{Ext}^1_{\R \text{MHS}}(\R, H) \to \R_{\geq 1}$ the height function on $\text{Ext}^1_{\R \text{MHS}}(\R, H)$.
\item Let $h_{\diamond, p,d}: \text{Ext}^1_{\R \text{MHS}}(\R, H) \to \mathbb{R}_{\geq 0}$ be defined by 
 \[ h_{\diamond, p,d}:=(\ln H_{\diamond,p,d}) .\] 
We call $h_{\diamond,p,d}$ the logarithmic height function on $\text{Ext}^1_{\R \text{MHS}}(\R, H)$.
\end{enumerate}
\end{definition} 
We remark that $h_{\diamond, \infty,d}^d$ the $d-$ power of the logarithmic height function  is given by a positive definite quadratic form $Q_{\diamond, p,d}$ on $\text{Ext}^1_{\R \text{MHS}}(\R, H)$. 
\begin{rem}
We consider a special case of the above discussion in the case $H_{\R}$ is the mixed $\R$-Hodge structure arising from the motive $M$. In this case, $H_{\R}$ is $V_{\infty}=H_{B}(M) \otimes \R$ in Bloch-Kato's paper $\cite{[BK]}$. Here $H_{B}(M)$ is the Betti realization of $M$. In this case, the above isomorphism can be written as
\[ \text{Ext}^1_{\R \text{MHS}}(\R, V_{\infty}) \cong H_{\R}^{\leq -1, \leq -1} =V_{\infty} \cap  \bigoplus_{p, q \leq -1, p+q=-d} H_{\C}^{p,q}.\]
In particular, we can define the height function $H_{\diamond, \infty,d}$ and the logarithmic height function $h_{\diamond, \infty,d}$ on $\text{Ext}^1_{\R \text{MHS}}(\R, V_{\infty})$.
\end{rem}
To define height function and logarithmic height function on $B(\R)$, we will construct a map from $B(\R)$ to $\text{Ext}^1_{\R \text{MHS}}(\R, V_{\infty})$. First, following section $2.1.10$ of $\cite{[Kato1]}$ we introduce the following notation. 
\begin{definition}
We define 
\[ E_{\infty}^{\R}=V_{\infty}^{\sigma=-1} \cap  \bigoplus_{p, q \leq -1, p+q=-d} H_{\C}^{p,q} .\]
By definition $E_{\infty}^{\R}$ is a subset of $\text{Ext}^1_{\R \text{MHS}}(\R, V_{\infty})$. 
\end{definition} 
First, we observe that there is a canonical map 
\[ p: B(\R)/B(\R)_{\text{cpt}}= D_{\infty}/(\Fil^{0} D_{\infty}+V_{\infty}^{\sigma=1}) \to E_{\infty}^{\R} .\] 
We have 
\[ D_{\infty}=(V_{\infty} \otimes \C)^{\sigma=1}=(V_{\infty}^{\sigma=1} \otimes \R) \oplus (V_{\infty}^{\sigma=-1} \otimes \R i) .\]
In particular, there is a projection map 
\[ D_{\infty} \to V_{\infty}^{\sigma=-1} \otimes \R i .\] 
By composing this map with the map $V_{\infty}^{\sigma=-1} \otimes \R i \to V_{\infty}^{\sigma=-1}$ sending $e \otimes i \to e$, we get a map $D_{\infty} \to V_{\infty}^{\sigma=-1}$. We then get a map 
\[p: D_{\infty} \to  E_{\infty}^{\R}:=( V_{\infty}^{\sigma=-1} \cap \bigoplus_{p, q \leq -1, p+q=-d} H_{\C}^{p,q}).\] 
It is easy to see that $(\Fil^{0} D_{\infty}+V_{\infty}^{+})$ belongs to the kernel of this map. Hence, $p$ induces a map 
\[ p: B(\R)/B(\R)_{\text{cpt}} \to E_{\infty}^{\R} .\]
Because $E_{\infty}^{\R}$ is a subset of $\text{Ext}^1_{\R \text{MHS}}(\R, V_{\infty})$, we can consider $p$ as a map from $B(\R)/B(\R)_{\text{cpt}}$ to $\text{Ext}^1_{\R \text{MHS}}(\R, V_{\infty})$. By composing $p$ with the projection map $B(\R) \to B(\R)/B(\R)_{\text{cpt}}$ we can define a canonical map 
\[ p: B(\R) \to \text{Ext}^1_{\R \text{MHS}}(\R, V_{\infty}) .\] 
Finally, by composing $p$ with the height function $H_{\diamond,\infty,d}$ and the logarithmic height function $h_{\diamond,\infty,d}$ on $\text{Ext}^1_{\R \text{MHS}}(\R, V_{\infty})$ we a height function and a logarithmic height function on $B(\R)$. We will still denote them by $H_{\diamond,\infty,d}$
and $h_{\diamond,\infty,d}$. 
By choosing a lift, we have an decomposition 
\[ B(\R)=\R^{s_{\infty}} \times B(\R)_{\text{cpt}} .\] 
By definition, $h_{\diamond,\infty,d}^{d}$ the $d$- power of the logarithmic height function on $B(\R)$ is given by a positive definite quadratic form $Q_{\diamond,\infty,d}$ on $\R^{s_{\infty}}$. Moreover, it is trivially $0$ on $B(\R)_{\cpt}$.  With this observation, the following statement is obvious. We remark that it is the archimedean analogue of Proposition $\ref{finite}$.
\begin{prop} \label{finite2}
For each positive real number $B$ 
\[ \mu(a \in B(\R)| H_{\diamond, p,d}(a) \leq B) <\infty .\] 
\end{prop} 
The following question seems natural. 

\begin{question}
Suppose $M=H^n(X)(r)$ where $X$ smooth projective variety over $\Q$. Can we describe $s_{\infty}$ from the geometry of $X$?
\end{question}

\subsection{Height function on adelic space $\prod_{p} B(\Q_p)$} \label{hadelic}
Having defined height functions and logarithmic height functions at all local places $B(\Q_p)$ we can now define a height function as well as logarithmic height function on the adelic space $\prod_{p} B(\Q_p)$. More precisely, we define
\[ H_{\diamond,d}: \prod_{p \leq \infty} B(\Q_p) \to \R_{\geq 1} ,\] 
by sending $a=(a_p) \in \prod_{p \leq \infty} B(\Q_p)$ to 
\[ H_{\diamond,d}(a)=\prod_{p \leq \infty} H_{\diamond,p,d}(a_p) .\] 
By Remark $\ref{interesting height functions}$, all but finitely many terms in this product are $1$ so the product makes sense. 
Similarly, we define the logarithmic height $h_{\diamond, d}=\ln H_{\diamond, d} $, for $a=(a_p)$
\[ h_{\diamond,d}(a)=\sum_{p \leq \infty} h_{\diamond,p,d}(a_p) .\] 
We can describe $h_{\diamond,d}$ concretely as follows. For $a \in \prod_{p} B(\Q_p)$ let $([a]_p)$ be the image of $a$ under the projection 
\[ \prod_{p} B(\Q_p) \to  \left(B(\R)/B(\R)_{\cpt} \right) \times \prod_{p<\infty} B(\Q_p)/A(\Q_p)=\R^{s_{\infty}} \times \prod_{p<\infty} \Z^{s_p} .\] 
Then 
\[ h_{\diamond,d}(a)=\sum_{p \leq \infty} h_{\diamond,p,d}([a]_p)= Q_{\diamond,d}([a]_{\infty})^{1/d}+\sum_{p<\infty} (\ln{p}) Q_{\diamond,p,d}([a]_p)^{1/d} .\]
With this description, the following statement is relatively straightforward. We provide a complete proof because we will use it later. 

\begin{prop} \label{finite3}
For each positive real number $B$, the set 
$$\{a \in \prod_{p \leq \infty} B(\Q_p)| H_{\diamond,d}(a) \leq B \},$$ 
is compact. In particular, 
\[ \mu(a \in \prod_{p \leq \infty} B(\Q_p)| H_{\diamond,d}(a) \leq B) <\infty .\] 
\end{prop} 
\begin{proof}
Let $S_{B}$ be the above set. Then 
$$S= \{a \in \prod_{p \leq \infty} B(\Q_p)| h_{\diamond,d}(a) \leq \ln(B)\}$$ 
Under the identification $B(\Q_p)=A(\Q_p) \times \Z^{s_p}$ and $B(\R)=B(\R)/B(\R)_{\cpt} \times B(\R)_{\cpt}$, it is the same as the following set 
\[ \left \{(a_{\infty} \times B(\R)_{\cpt} ) \times \prod_{p <\infty} ([a]_p \times A(\Q_p)) \right \}, \] 
where 
\[ Q_{\diamond,d}([a]_{\infty})^{1/d}+\sum_{p<\infty} (\ln{p}) Q_{\diamond,p,d}([a]_p)^{1/d} \leq \ln{B}   .\] 
For each $p$, $Q_{\diamond,p,d}$ is a positive definite quadratic form, there are only finitely many $[a]_p \in \Z^{s_p}$ verifying the above inequality. We also know that $A(\Q_p)$ is compact for all $p<\infty$. Therefore, $S$ is compact. Moreover, if we equip $\R^{s_{\infty}} \times \prod_{p<\infty} {\Z}^{s_p}$ with the product measure, then it is easy to see that \[ \mu(S)=\mu(B(\R)_{\cpt}) \times \prod_{p< \infty} \mu(A(\Q_p)) \times \mu (T_{B}), \] 
where 

\[  T_{B}= \left\{([a]_p) \in \R^{s_{\infty}} \times \prod_{p<\infty} {\Z}^{s_p} \big| Q_{\diamond,d}([a]_{\infty})^{1/d}+\sum_{p<\infty} (\ln{p}) Q_{\diamond,p,d}([a]_p)^{1/d} \leq \ln{B} \right\} .\] 
\end{proof}

\subsection{Height function on $B(\Q)$} \label{adelic}
Having defined height functions on local spaces $B(\Q_p)$, we now define a height function on the global space $B(\Q)$. First, we recall that $B(\Q)$ is defined by the following fiber product
 diagram
\[\xymatrix@C+1pc{
       &B(\Q) \ar[d] \ar[r] &H_{g}^1(\Q,\widehat{\Theta})  \ar[d]  \\
       &\Psi \ar[r] &H_{g}^1(\Q, \widehat{\Theta} \otimes \textbf{A}_{\Q}^{f})   \\
 }\]
We also recall that $A(\Q)$ is defined to be the fiber product
 diagram
\[\xymatrix@C+1pc{
       &A(\Q) \ar[d] \ar[r] &H_{f}^1(\Q,\widehat{\Theta})  \ar[d]  \\
       &\Psi \ar[r] &H_{g}^1(\Q, \widehat{\Theta} \otimes \textbf{A}_{\Q}^{f})   \\
 }\]
\begin{rem}
Philosophically, we think of $\Psi$ as $\text{Ext}^{1}_{\mathcal{M M}}(1, M)$ where $\mathcal{M M}$ is the category of mixed motives with $\Q$-coefficients. By definition, an element of $B(\Q)$ is of the form $(a,b)$ where $a \in \Psi$ and $b \in H_{g}^1(\Q,\widehat{\Theta})$ such that the images of $a$ and $b$ in $H_{g}^1(\Q, \widehat{\Theta} \otimes \textbf{A}_{\Q}^{f})$ coincide. By definition, $a$ corresponds to a mixed motives $M_a$ that fits into a short exact sequence 
\[ 0 \to M \to M_a \to \Q \to 0 .\] 
Similarly $b$ corresponds to an extension of $\widehat{\Z}$-modules with a continuous $\Gal(\overline{\Q}/\Q)$-action 
\[ 0 \to \widehat{\Theta} \to \widehat{\Theta}_{b} \to \widehat{\Z} \to 0.\] 
The compatibility of $a$ and $b$ means that $(M_a, \widehat{\Theta}_{b})$ is a motive with $\Z$-coefficients. Therefore, we should think about $B(\Q)$ as the extension group $\Ext^1(\Z, (M, \Theta))$ in the category of motives with $\Z$-coefficients. Similarly, we should think about $A(\Q)$ the extension group $\Ext^1(\Z, (M, \Theta))$ in the category of motives with $\Z$-coefficients with local conditions at non-archimedean places.
\end{rem} 
In what follows, we will use the following convention. If $M$ is a finitely generated abelian group, then $M_{\tor}$ is the torsion subgroup of $M$ and $M_{\text{free}}=M/M_{\tor}$. By definition, $M_{\text{free}}$ is a free $\Z$-module of finite rank.

By definition, $A(\Q) \subset B(\Q)$ and both are finitely generated abelian groups. Moreover we have  $$A(\Q) \otimes \widehat{\Z}=H^1_{f}(\Q, \widehat{\Theta}), A(\Q) \otimes \Q=\Phi,  B(\Q) \otimes \widehat{\Z}=H^1_{g}(\Q, \widehat{\Theta}), B(\Q) \otimes \Q=\Psi.$$
We have a global analogue of Lemma $\ref{Hg/Hf}$ that $B(\Q)/A(\Q)$ is a free abelian group of finite rank. Conjectures of Beilinson, Bloch, Kato,  and Jannsen (see $\cite{[Beilinson]}$, $\cite{[BK]}$,  and $\cite{[Jannsen2]}$) predict that 
\begin{enumerate}
\item The regulator map induces an isomorphism $$A(\Q) \otimes \R = \Phi \otimes \R \cong B(\R)/B(\R)_{\text{cpt}}= E_{\infty}^{\R}.$$
In particular, there is a map $A(\Q) \to B(\R)/B(\R)_{\cpt}$. One can lift this map to be a map $h: A(\Q) \to B(\R)$. In particular, the image of $A(\Q)_{\tor}$ belongs to $B(\R)_{\cpt}$. It is easy to see that we can choose a splitting $B(\R)=B(\R)_{\cpt} \times B(\R)/B(\R)_{\cpt}$ and $A(\Q)=A(\Q)_{\text{free}} \times A(\Q)_{\tor}$ in such a way that $h=(h_1,h_2)$ where $h_1: A(\Q)_{\text{free}} \to B(\R)/B(\R)_{\cpt}$ and $h_2: A(\Q)_{\tor} \to B(\R)_{\cpt}$. The above isomorphism implies that $h_1$ is an embedding and it identifies $A(\Q)_{\text{free}}$ with a lattice in $B(\R)/B(\R)_{\cpt}$.

\item The canonical projection map 
\[B(\Q)/A(\Q) \to \prod_{p<\infty} B(\Q_p)/A(\Q_p), \]
 is an isomorphism.   
\end{enumerate}  
As a corollary, we conclude that $\rank_{\Z} A(\Q) = \rank_{\Z} \Phi=B(\R)/B(\R)_{\text{cpt}}=s_{\infty}$ and $s_{\text{fin}}=\rank(B(\Q)/A(\Q))=\sum_{p<\infty} s_p$. Moreover, it is easy to we can define a splitting (up to a choice of basis) 
\[ B(\Q)=A(\Q) \times \Z^{s_{\text{fin}}}, \] 
in such a way that $\Z^{s_{\text{fin}}}$ belongs to the kernel of the regulator map $R_{\infty}: \Psi \otimes \R \to B(\R)/B(\R)_{\text{cpt}}$. We then have 
\[ \Z^{s_{\fin}} \cong \prod_{p<\infty} B(\Q_p)/A(\Q_p) =\prod_{p <\infty} \Z^{s_p}. \]

We define a height function and a logarithmic height function on $B(\Q)$. As explained above, there is an evident map
\[ B(\Q) \to \prod_{p \leq \infty} B(\Q_p) .\] 
By composing this map with the height function and the logarithmic height functions defined in section $\ref{adelic}$, we can define a height function and a logarithmic height function on $B(\Q)$ as follows
\begin{definition}
$H_{\diamond,d}: B(\Q) \to \R_{\geq 1}$ is defined by 
\[ H(a)=\prod_{p \leq \infty} H_{\diamond,p,d} (a_p) .\] 
Similarly, $h_{\diamond,d}: B(\Q) \to \R_{\geq 0}$ is defined to be 
\[ h_{\diamond,d}=\ln(H(a))=\sum_{ p \leq \infty} h_{\diamond,p,d}(a_p) =Q_{\diamond,\infty,d}(a_{\infty})^{1/d}+ \sum_{p<\infty} (\ln{p}) Q_{\diamond,p,d}(a_p)^{1/d} .\]
\end{definition}
As explained in Section $\ref{hadelic}$,  we can describe logarithmic height function $h_{\diamond,p,d}$ concretely as follows. For $a \in B(\Q)$ we write $a=\sum_{p} [a]_{p}$ under the splitting: $B(\Q)=A(\Q) \times \prod_{p} \Z^{s_p}$, then $([a]_{\infty})_{p} \in A(\Q_p)$ for all $p<\infty$, hence $h_{\diamond,p,d}(b_{\infty})=0$. Similarly, if $p<\infty$, then $R_{\infty}([a]_{p})=0$ in $B(\R)/B(\R)_{\text{cpt}}$, hence $h_{\diamond,\infty,d}([a]_{p})=0$. From this observation, we conclude that 
\[ h_{\diamond,d}(a)=\sum_{p \leq \infty} h_{\diamond,p,d}([a]_p) .\] 
We also note that $B(\Q)_{\tor}=A(\Q)_{\tor}$ is a finite group, and if $b \in B(\Q)_{\tor}$ then $h_{\diamond,d}(a)=0$.  Note that there is a splitting  $A(\Q)=A(\Q)_{\tor} \times A(\Q)_{\free}$. Hence, there is also a splitting 
\[ B(\Q)= B(\Q)_{\text{free}} \times B(\Q)_{\tor}= A(\Q)_{\tor} \times A(\Q)_{\text{free}} \times \prod_{p<\infty} \Z^{s_p}=A(\Q)_{\tor} \times \prod_{p \leq \infty} \Z^{s_p} .\]
With this observation, the following statement is a direct corollary of $\ref{finite3}$. 

\begin{prop} \label{finite4}
For each positive integer $B$, the set 
\[ \{b \in B(\Q)|H_{\diamond,d}(b)<B \} ,\]
is finite.
\end{prop}
\begin{proof}
We have 
\[ \{b \in B(\Q)|H_{\diamond,d}(b)<B \}= B(\Q)_{\tor} \times  \{b \in B(\Q)_{\free}|H_{\diamond,d}(b)<B \} .\]
Under the above identification, the second component can be described as 
\[ \{([a]_p) \in \prod_{p \leq \infty} \Z^{s_p} |Q_{\diamond,\infty,d}(a_{\infty})^{1/d}+ \sum_{p<\infty} (\ln{p}) Q_{\diamond,p,d}(a_p)^{1/d} \leq \ln{B} \}. \]
This is evidently a finite set. Hence $\{b \in B(\Q)|H_{\diamond,d}(b)<B \}$ is a finite set as well.

\end{proof}

\section{Proof of main theorem} \label{proof1}
We give a proof for Theorem $\ref{theorem1}$. First, we recall the definition of $\Tam(M)$ as defined in $\cite{[BK]}$
\[ \Tam(M)=\mu \left(\prod_{p} A(\Q_p)/A(\Q) \right) .\] 
Because $A(\Q_p)$ and $A(\R)_{\cpt}$ is compact for all $p<\infty$, this can be rewritten as 
\[ \Tam(M)= |A(\Q)_{\tor}|^{-1} \times \mu((B(\R)/B(\R)_{\cpt})/A(\Q)_{\free})\times \mu(A(\R)_{\cpt}) \times \prod_{p} \mu( A(\Q_p) .\] 

We know from previous discussions that the set $\{x \in B(\Q)| H_{\diamond,d}(x) \leq B \}$ is finite and $ \mu(a \in \prod_{p} B(\Q_p)| H_{\diamond,d}(a) \leq B) <\infty$. Hence, the quotient  
\[ \frac{ \# \{a \in B(\Q)| H_{\diamond,d}(b) \leq B \}}{\mu(a \in \prod_{p} B(\Q_p)| H_{\diamond,d}(a) \leq B)} =\frac{C_1}{C_2}\]
makes sense. Here 
\[ C_1= \# \{a \in B(\Q)| H_{\diamond,d}(b) \leq B \} ,\]
and 
\[ C_2= \mu(a \in \prod_{p} B(\Q_p)| H_{\diamond,d}(a) \leq B) .\] 
By the discussions around Proposition $\ref{finite3}$ and Proposition $\ref{finite4}$ we have 
\[ C_1=|A(\Q)_{\tor}| \times N_{B}, \]
where 
\[
N_{B}=\# \{([a]_p) \in \prod_{p \leq \infty} \Z^{s_p} |Q_{\diamond,\infty,d}(a_{\infty})^{1/d}+ \sum_{p<\infty} (\ln{p}) Q_{\diamond,p,d}(a_p)^{1/d} \leq \ln{B} \} .\] 
We also have  
\[ C_2= \mu(B(\R)_{\cpt}) \times \prod_{p< \infty} \mu(A(\Q_p)) \times \mu (T_{B}), \]
 where 
 \[  T_{B}= \left\{([a]_p) \in \R^{s_{\infty}} \times \prod_{p<\infty} {\Z}^{s_p} \big| Q_{\diamond,\infty, d}([a]_{\infty})^{1/d}+\sum_{p<\infty} (\ln{p}) Q_{\diamond,p,d}([a]_p)^{1/d} \leq \ln{B} \right\} .\] 
In summary, we have
\[ \dfrac{C_1}{C_2}=\dfrac{\mu((B(\R)/B(\R)_{\cpt})/A(\Q)_{\free}) \times N_{B}}{\Tam(M) \times \mu(T_{B})}. \] 
Therefore, Theorem $\ref{theorem1}$ is equivalent to 
\[ \lim_{B \to \infty} \dfrac{N_{B}}{\mu(T_B)}=(\mu((B(\R)/B(\R)_{\cpt})/A(\Q)_{\free}))^{-1} .\] 
Once we fix a basis and rescale the measure, we can assume that $A(\Q)_{\free}$ is the standard basis for $B(\R)/B(\R)_{\cpt}$ and that $\mu(B(\R)/B(\R)_{\cpt}/ A(\Q)_{\free})=1$. The above statement then becomes $\lim_{B \to \infty} \dfrac{N_{B}}{\mu(T_B)}=1$, which is a consequence of the following general statement.

\begin{prop}
Let $f=(f_1,f_2):\R^{r_1+r_2}=\R^{r_1} \times \R^{r_2} \to \R$. Suppose for each $i=1,2$ the following conditions are satisfied
\begin{enumerate}
\item $f_i(x) \geq 0$ for all $x$. Moreover, $f_i(x)=0$ if and only if $x=0$. 
\item $f_i(x)$ is continuous.
\item There exists a positive real number $c$ such that $f(\lambda x)=|\lambda|^{c} f(x)$ for all $x \in \R^{r_i}$ and $\lambda \in \R$. 
\end{enumerate}
We equip $\Z^{r_1}$ with the counting measure $d \mu_{0}$ and $\R^{r_1}, \R^{r_2}$ the usual Lebesgue measures $d \mu_{1}, d \mu_{2}$. Let  
\[ I(B)=((n,y) \in \Z^{r_1} \times \R^{r_2}| f(n,y)=f_1(n)+f_2(y) \leq B ) ,\] 
and 
\[ V(B)=((x,y) \in \R^{r_1} \times \R^{r_2}| f_1(x)+f_2(y) \leq B). \] 
Then 
\[ \lim_{B \to \infty} \dfrac{\mu(I(B)}{\mu(V(B))}=1 .\] 
\end{prop} 
\begin{proof}
By property $(iii)$ of $f_1,f_2$, $f(x,y) \leq B $ is equivalent to 
\[ f_1 \left(\frac{x}{B^{1/c}} \right) +f_{2}\left(\frac{y}{B^{1/c}} \right) \leq 1 .\] 
From this observation 
\[ \mu(V(B))=(B^{1/c})^{r_1+r_2} \mu(S_1) ,\]
where 
\[S_1=((x,y) \in \R^{r_1} \times \R^{r_2}| f_1(x)+f_2(y) \leq 1) .\] 
Let $D=\{x \in \R^{r_1}| f_1(x)  \leq 1 \}$. For each $x$ in $D$, we define 
\[ g(x)=\int_{f_2(y) \leq 1-f(x)} 1 d \mu_2 .\] 
By our assumption, $g(x)$ is an integrable function on $D$. We have 
\[ \mu (S_1)=\int_{\R^{r_1+r_2}} 1_{S_1} .\]
By Fubini's theorem, we have 
\[ \mu(S_1)=\int_{D} g(x) d \mu_{1} .\]
We also have 
Similarly, 
\[ I(B)=\{(n,y) \in \Z^{r_1} \times \R^{r_{2}}| f_1 \left(\frac{n}{B^{1/c}} \right) +f_{2}\left(\frac{y}{B^{1/c}} \right) \leq 1 \} .\] 
By a similar argument as above we have 
\begin{align*}
\mu(I(B)) &=(B^{1/c})^{r_2} \times \sum_{f_1(n) \leq B } \left( \int_{f_{2}(y) \leq 1-f_1(\frac{n}{B^{1/c}})} 1 d \mu_{2} \right) \\
          &=(B^{1/c})^{r_2} \times \sum_{a \in \frac{1}{B^{1/c}} \Z^{r_1} \cap D} g(a) \\
          &=(B^{1/c})^{r_1+r_2} \times \sum_{a \in \frac{1}{B^{1/c}} \Z^{r_1} \cap D} g(a) \left(\frac{1}{B^{1/c}} \right)^{r_1} .
\end{align*}
We observe that second term is nothing but the Riemann sum of $g$ with respect to the partition $\frac{1}{B^{1/c}} \Z^{r_1} \cap D$. Since $g(x)$ is integrable on $D$, we have 
\[ \lim_{B \to \infty} \left[ \sum_{a \in \frac{1}{B^{1/c}} \Z^{r_1} \cap D} g(a) \left(\frac{1}{B^{1/c}} \right)^{r_1} \right] =\mu(S_1) .\] 
The above statement follows immediately from this. 

\end{proof}

\subsection{Some lemmas on counting integer points}
We discuss some simple lemmas about counting integer points in bounded domains in $\R^2$.

\begin{lem}(Euler's summation formula, see \cite[Theorem 3.1]{[Apostol]}) \label{lem: integer}
Let $f$ be a continuously differentiable function on $[y,x]$. Then 
\[ \sum_{y<n \leq x} f(n) = \int_{y}^{x} f(u)dt+ \int_{y}^{x} \{u \} f'(u)du +f(x)\{x \}-f(y) \{y \}.\]
Here $\{x \}=x-[x]$ is the fractional part of $x$.
\end{lem} 

We will use the following lemma repeatedly. 
\begin{lem} \label{lem: counting}
Let $s,t$ be positive integers. For each positive real number $X$, define
 \[ d(X)= \# \{(m,n) \in \N \times \N| a m^{1/s}+b n^{1/t} \leq X \} . \] 
Then 
\[ d(X) \sim \dfrac{1}{ {s+t\choose t}} \dfrac{1}{a^s b^t} X^{s+t} .\] 
\end{lem} 
\begin{proof}
We see that $a m^{1/s}+b n^{1/t} \leq X$ is equivalent to 
\[ m \leq \left(\dfrac{X-bn^{1/t}}{a} \right)^s .\] 
Therefore we conclude that $d(X) \sim d_1(X)$ where
\[ d_1(X)= \sum_{0<n \leq (X/b)^t}  \left(\dfrac{X-bn^{1/t}}{a} \right)^s .\]
Let $f(u)=\left(\dfrac{X-bu^{1/t}}{a} \right)^s$, then by Lemma $\ref{lem: integer}$ we have 
\[ d_1(X)= \int_{0}^{(X/b)^t} f(u)du- \int_{0}^{(X/b)^t} (u-[u])f'(u)du .\]
Because $0\leq u-[u]<1$ we have 
\[ \left| \int_{0}^{(X/b)^t} (u-[u])f'(u)du \right| \leq \int_{0}^{(X/b)^t} |f'(u)|du=\left(\frac{X}{a} \right)^s .\]
We will estimate the first term 
\[ \int_{0}^{(X/b)^t} \left(\dfrac{X-bu^{1/t}}{a} \right)^s du .\] 
By changing variable $u=(X/b)^t v^t$, we see that this integral is equal to 
\[ \frac{X^{s+t}}{a^s b^t} \int_{0}^{1} t (1-v)^s v^{t-1} dv .\] 
The above integral is a Gamma function; and by checking the table of values we conclude that  
\[ d(X) \sim \dfrac{1}{ {s+t\choose t}} \dfrac{1}{a^s b^t} X^{s+t}, \text{as} \quad X \to \infty .\] 
\end{proof}

\section{Some computations with mixed Tate motives}
In this section, we give a proof for Theorem \ref{thm:theorem_2}. Let us first recall some notations and conventions. Let $m,n$ be two natural number such that $m-n \geq 2$, $m$ is even and $n$ is odd. In what follows, let us fix a mixed motive $M=(V,D)$ with $\Z$-coefficients with graded quotients $\Z(m), \Z(n)$. We will compute the numbers of $x \in \Ext^1(\Z,(V, D))$ with height bounded by a positive number $B$. We then apply our calculations to the concrete case where $m=12$ and $n=3$. The reason we include this concrete example is that we want to compare it with the situation in the next section where $(V,D)$ also moves among mixed motives with $\Z$-coefficients and with graded quotients $\Z(12), \Z(3), \Z$. 
 
For each prime $p$, let us denote $D_p$ the $p$-adic component of $D \otimes \widehat{\Z}$ and put $V_p=D_p \otimes \Q_p$. By definition $D_p$ is a Galois stable $\Z_p$ lattice inside $V_p$. Let us also recall the definition of $\Sha$ of  a mixed motive.  For a fixed prime $p$ and for each place $v$ of $\Q$, we have special subgroups of $H^{1}_{f}(\Q_{v}, D_p), H^1_{f}(\Q_{v}, V_p)$ and $H^1_{f}(\Q_{v}, V_p/D_p)$ of $H^1(\Q_v, D_p)$, $H^1(\Q_{v}, V_p)$ and $H^1_{f}(\Q_{v}, V_p/D_p)$ respectively. For example: 

\[ H^1_{f}(\Q_v, V_p) = \begin{cases} H_{ur}^1(\Q_v,V_p) &\mbox{if } v \nmid p \\
\ker(H^1(\Q_v, V_p) \to H^1(\Q_v, V_p \otimes B_{\text{cris}})) & \mbox{if } v \mid p. \end{cases}  \]

The Selmer group of $D_p$ is defined to be 
\[ H^{1}_{f}(\Q, V_p/D_p)=\{ x \in H^1(\Q, V_p/D_p)| x_{v} \in H^{1}_{f}(\Q_{v}, V_p/D_p), \forall v \} .\] 
Similarly, we define 
\[ H^1_{f}(\Q, V_p) =\{ x \in H^1(\Q, V_p)| x_{v} \in H^{1}_{f}(\Q_{v}, V_p), \forall v \} .\]
By definition, we have a canonical map $H^{1}_{f}(\Q,V_p) \to H^1_{f}(\Q, V_p/D_p)$. 
\begin{definition}
Define $\Sha (D_p)$ to be the cokernel of the map 
\[ H^{1}_{f}(\Q,V_p) \to H^1_{f}(\Q, V_p/D_p). \] 
\end{definition} 
We define the global $\Sha(D)$ to be: 
\[ \Sha(D)=\prod_{p} \Sha(D_p) .\] 
It is conjectured to be a finite group. In our special case, we recall that $V_p$ fits into an exact sequence 
\[ 0 \to \Q_p(m) \to V_p \to \Q_p(n) \to 0 .\] 
We have the following simple lemma. 

\begin{lem} \label{lem:crystalline}
As a Galois representation of $\Gal(\overline{\Q_v}/\Q_v)$, $V_p$ is unramifed if $v \nmid p$ and it is crystalline if $v \mid p$. Moreover, 

\begin{enumerate} 
\item If $v \nmid p$ then $H^{i}(\Q_v, V_p)=0$ for $i=0,1,2$. In particular, $H^1_{f}(\Q_v, V_p)=0.$
\item If $ v \mid p$ then $H^1_{f}(\Q_{v}, V_p)= H^1(\Q_{v}, V_p)$. Furthermore, $H^0(\Q_v, V_p)=H^2(\Q_v, V_p)=0$.
\end{enumerate}
\end{lem}
\begin{proof}
The statements that $V_p$ is unramifed if $v \nmid p$ and it is crystalline if $v \mid p$ follow from the fact that $H^1(\Q_v, \Q_p(m-n))=H^1_{f}(\Q_v, \Q_p(m-n))$ (the case $v \nmid p$ is discussed in \cite[\S 3]{[Sch]}, the case $v|p$ is discussed in \cite[Example 3.9]{[BK]}). For the first statement, we note that the short exact sequence 
\[ 0 \to \Q_p(m) \to V_p \to \Q_p(n) \to 0 ,\] 
induces a long exact sequence of cohomology groups 
\[ \ldots \to H^{i}(\Q_v, \Q_p(m)) \to H^{i}(\Q_v, V_p) \to H^i(\Q_v, \Q_p(n)) \to H^{i+1}(\Q_v, \Q_p(m)) \to \ldots \] 
Since $m, n \geq 2$, $H^i(\Q_v, \Q_p(m))=H^i(\Q_v, \Q_p(n))=0$ for all $i \geq 0$ (see \cite[\S3]{[Sch]}). Therefore, we conclude that $H^i(\Q_v, V_p)=0$ for all $i \geq 0.$ 

The second statement follows from a similar arguments plus the calculations in \cite[Example 3.9]{[BK]}. 

\end{proof}

To prepare the the proof the main Theorem \ref{thm:theorem_2}, we recall some fundamental results about the algebraic $K$-groups $K_n(\Z)$.  
\begin{lem} \label{lem:Soule} (The Tamagawa number conjecture for $\Z(m)$, see $\cite{[Bloch-Kato]}$)
Let $m$ be a positive odd integer. Then, there exists an element $b_{m}$ such that 
\begin{enumerate}
\item $b_{m}$ is the generator of $K_{2m-1}(\Z)$ modulo torsion. 
\item $r_{\infty}(b_m)=2^a \dfrac{(m-1)! \zeta(m)}{\Sha(m)}$. 
where $a$ is an unknown number. 
\end{enumerate}
\end{lem}
We also use the following lemma about the precise structure of $K_n(\Z)$. 

\begin{lem} (see \cite[Theorem 10.1]{[Weibel]})  \label{lem:k_n}
\begin{enumerate}
    \item If $m=8k+1$, then $K_m(\Z)=\Z \oplus \Z/2.$
    \item If $m=8k+5$, then $K_{m}(\Z)=\Z$.
\end{enumerate}
\end{lem} 
By Lemma \ref{lem:Soule} and Lemma \ref{lem:k_n}, we conclude that if $m$ is a positive integer of the form $8k+1$, then $b_m$ is a generator of $K_{m}(\Z)$ modulo the $2$-torsion part. In the case $m=8k+5$, $K_m(\Z)$ is an infinite cyclic group generated by $b_m$.

For simplicity we will write $b_{m,p}$ the image of $b_{m}$ under the Chern class map 
\[ K_{2m-1}(\Z) \otimes \Z_p \to H^1(\Z[1/p], \Z_p(m)) .\] 
As a consequence of Quillen-Lichtenbaum theorem (see \cite{[Bloch-Kato]} for more detailed discussions), the above Chern class map is an isomorphism for $p$ odd. Consequently, $b_{m,p}$ is the generator of $H^1(\Z[1/p], \Z_p(m))$ for all $p$ odd.

We also have the following proposition. 
\begin{prop}
Let $V_p, D_p$ be defined as above. Then 
\[ \Sha(D_p)= \ker \left(H^2(\Z[1/p],D_p) \to H^2(\Q_p,D_p) \right) .\]

\end{prop} 
\begin{proof}
By Lemma \ref{lem:crystalline}, we have  
\[ H^1_{f}(\Q, V_p)=H^1(\Z[1/p], V_p) .\]
The long exact sequence associated with the short exact sequence $$0 \to D_p \to V_p \to V_p/D_p \to 0$$ gives 
\begin{align*}
H^{1}_{f}(\Q_{v}, V_p/D_p) &=\text{Im}\left(H^1(\Q_{v}, V_p) \to H^1(\Q_v, V_p/D_p) \right) \\
                           & = \ker(H^1(\Q_{v}, V_p/D_p) \to H^2(\Q_{v}, D_p)) .
\end{align*} 
Therefore, we have 
\begin{align*}
H^1_{f}(\Q, V_p/D_p) &=\ker(H^1(\Q, V_p/D_p) \to \bigoplus_{v} H^2(\Q_v,D_p)) \\
                        &=\ker \left(H^1(\Z[1/p], V_p/D_p) \to H^2(\Q_p,D_p) \right) .
\end{align*}
We also observe that the short exact sequence 
\[ 0 \to \Q_p(m) \to V_p \to \Q_p(n) \to 0 ,\] 
induces the following exact sequence 
\[ H^2(\Z[1/p], \Q_p(m)) \to H^2(\Z[1/p], V_p) \to H^2(\Z[1/p], \Q_p(n)) .\] 
By the Quillen-Lichtenbaum isomorphism (see \cite{[Bloch-Kato]}) and \cite[Theorem 10.1]{[Weibel]}, we know that $H^2(\Z[1/p], \Q_p(m))=H^2(\Z[1/p], \Q_p(n))=0$. Therefore, $H^2(\Z[1/p], V_p)=0$ as well. In summary, we have the following commutative diagram
\[\begin{tikzcd}
       H^1(\Z[1/p],V_p) \ar[r] \ar[d]  &H^1(\Z[1/p],V_p/D_p) \ar[r, twoheadrightarrow] \ar[d] &H^2(\Z[1/p],D_p) \ar[d]   \\
   H^1_{f}(\Q, V_p/D_p) \ar[r, hook] &H^1(\Z[1/p],V_p/D_p) \ar[r] &H^2(\Q_p,D_p)).
   \end{tikzcd} \]
By snake lemma, we see that 
\[ \Sha(D_p)= \ker \left(H^2(\Z[1/p],D_p) \to H^2(\Q_p,D_p) \right) .\]
\end{proof}
We want to relate $\Sha(D_p)$ with $\Sha(\Z_p(n))$ and $\Sha(\Z_p(m))$. For simplicity, we will denote the later groups by $\Sha_{p}(n)$ and $\Sha_{p}(m)$ respectively.  Taking long exact sequences associated to 
\[ 0 \to \Z_p(m) \to D_p \to \Z_p(n) \to 0,\] 
we have the following commutative diagram
\[\xymatrix{
       0 \ar[r] &H^2(\Z[1/p],\Z_p(m))/ \text{im}(\delta) \ar[r] \ar[d]  &H^2(\Z[1/p],D_p) \ar[r] \ar[d] &H^2(\Z[1/p],\Z_p(n)) \ar[d] \ar[r]  &0  \\
       0 \ar[r] &H^2(\Q_p, \Z_p(m))/\text{im}(\delta) \ar[r] &H^2(\Q_p,D_p) \ar[r] &H^2(\Q_p,\Z_p(n)) \ar[r] & 0   \\
 }\]
here $\delta$ are the connecting homomorphisms 
\[ \delta: H^1(\Z[1/p],\Z_p(n)) \to H^2(\Z[1/p], \Z_p(m)) \]
and 
\[ \delta: H^1(\Q_p, \Z_p(n)) \to H^2(\Q_p, \Z_p(m)).\] 
Note that if we think about $D_p$ as an element $[a]$ in $\Ext^1(\Z_p(n), \Z_p(m))$ then $\delta$ is the cup product of the $H^1(\Z[1/p],\Z_p(n))$ with $[a]$. Additionally, we note that the map $H^2(\Z[1/p], \Z_p(m)) \to H^2(\Q_p, \Z_p(m))$ is the zero map for $p>2$. This can be seen by looking at the long exact sequence of Poitou-Tate
\[ 0 \to H^0(\Z[1/p], \Q_p/\Z_p(1-m))  \to H^0(\Q_p, \Q_p/\Z_p(1-m)) \to H^2(\Z[1/p], \Z_p(m))^{\vee} \] 
Since $\Gal(\Q_p(\zeta_{p^{\infty}})/\Q_p)= \Gal(\Q(\zeta_{p^{\infty}})/\Q)$, we conclude that the map 
\[H^0(\Z[1/p], \Q_p/\Z_p(1-m))  \to H^0(\Q_p, \Q_p/\Z_p(1-m)), \]
is an isomorphism. Consequently, by local Tate duality, we see that the map $H^2(\Z[1/p], \Z_p(m)) \to H^2(\Q_p, \Z_p(m))$ is the zero map. By snake lemma, we have the short exact sequence 
\begin{equation}  \label{eq:Sha}
0 \to H^2(\Z[1/p],\Z_p(m))/ \text{im}(\delta) \to \Sha(D_p) \to \Sha(\Z_p(n)) \to 0, 
\end{equation}
where $\delta$ is the connecting homomorphism 
\[ \delta: H^1(\Z[1/p], \Z_p(n)) \to H^2(\Z[1/p], \Z_p(m)) .\]
Note that because $H^1(\Z[1/p], \Z_p(n))$ is a free $\Z_p$-module of rank $1$, the $\text{im}(\delta)$ is a cyclic subgroup of $H^2(\Z[1/p], \Z_p(m))$. Suppose that the order of $\text{im}(\delta)$ is $p^a$. Note that we have an exact sequence 
\[ H^1(\Z[1/p],\Z_p(m)) \hookrightarrow H^1(\Z[1/p], D_p) \to H^1(\Z[1/p], \Z_p(n)) \xrightarrow{\delta} H^2(\Z[1/p], \Z_p(m)) \] 
Hence, 
\[ \ker(\delta)=p^a H^1(\Z[1/p], \Z_p(n))=|\text{im}(\delta)| H^1(\Z[1/p], \Z_p(n)).\] 
Let $|\delta|=\prod_{p} |\text{im}(\delta)_{p}|$. As before, we can see that 
\[ X= \Z |\delta| b_{n} \oplus T, \]
where $b_{n}$ is the generator of $K_{2n-1}(\Z)$ described in Lemma $\ref{lem:Soule}$ and $$T=H^1(\Z[1/p], \Z_p(m)).$$ By Mazur-Wiles theorem (see \cite[Theorem 8.7]{[Weibel]}), we have 
\[ \#T = \prod_{p} \# H^1(\Z[1/p], \Z_p(m))= \dfrac{\# \Sha(m)}{\zeta(1-m)} .\] 
Also, by the exact sequence $\ref{eq:Sha}$ we have 
\[ \# \Sha(D)=\dfrac{\# \Sha(m) \# \Sha(n)}{|\delta|} .\]
By Lemma \ref{lem:crystalline}, the infinite place is the only place that contributes to the height function $H_{*, \diamond}$. Therefore, the order of the set $\{ x \in X| H_{*, \diamond}(x)<B \}$ is given by
\[ |T| \times \\
 \# \{z \in \Z| \exp(|r_{\infty}(ub_{m-n})|^{1/m-n}+|z r_{\infty}(|\delta| b_{n})|^{1/n}) \leq B \}  .\]
By Lemma \ref{lem: counting} and Lemma \ref{lem:Soule}, this number is asymptotically equal to 
\begin{equation*}
\begin{split}
|T| \dfrac{\log(B)^n}{|\delta| r_{\infty}(b_n)} &= 2^{t} \dfrac{\# \Sha(m)}{\zeta(1-m)} \dfrac{\#\Sha(n)}{|\delta| (n-1)! \zeta(n)} \log(B)^n \\  &=2^t \dfrac{\# \Sha(D)}{(n-1)! \zeta(n) \zeta(1-m)} \log(B)^n. 
\end{split} 
\end{equation*}
Note that we do not know the exact power of $2$ in this formula. 

\begin{rem}
I am thankful to the referee for asking whether we can determine the power of $2$ precisely in the above calculation. This could, in principle, be done if we know that the Quillen-Lichtenbaum isomorphism and the Mazur-Wiles formula are true at the $p=2$.

\end{rem}

In particular when $m=12$ and $n=3$, we have the following.
\begin{cor} \label{cor:fixed_quotient}
Let $M=(V,D)$ be a fixed motive with $\Z$ coefficients with graded quotients $\Z(12), \Z(3)$. Let $X=\Ext^1(\Z,D)$. Then 
\[ \# \{x \in X| H_{\star, \diamond} (x) \leq B \} \sim 2^{a} \dfrac{\Sha(D)}{\zeta(3) \zeta(-11)} \log(B)^3 . \] 
\end{cor}

\section{Heights of motives with graded quotients \\ $\Z(12), \Z(3), \Z$.}
In this section, we answer Kato's questions about the number of mixed motives with fixed graded quotients $\Z(12), \Z(3), \Z$ and bounded heights. Unlike the situation in Corollary \ref{cor:fixed_quotient} discussed in the previous section, the mixed motives with $\Z$-coefficients and with graded quotients $\Z(12), \Z(3)$ will also vary. More precisely, if we denote by $X$ the set of mixed motives with $\Z$-coefficients and with graded quotients $\Z(12), \Z(3), \Z$ then 
\[ X= \bigcup_{M \in \Ext^1(\Z(3), \Z(12))} \Ext^1(\Z, M) .\] 

For the definition of heights in this case, see $\cite{[Kato1]}$ for more details. We note that in this case, only archimedean places contribute to heights. Therefore, this problem is purely a problem of finding a ``basis" for $X$. The rest is a basic counting problem.

\subsection{Proof of theorem $\ref{theorem3}$}
First of all, we recall the following important fact. 

\begin{lem} \label{cup}
\noindent
\begin{enumerate}
\item We have the following 
\begin{equation*}
    H^2(\Z[1/p],\Z_p(12)) = \begin{cases}
               0               & \text{if} \quad p \neq 691 \\
               \Z/691 & \text{if} \quad p=691. 
           \end{cases}
\end{equation*}

\item The cup product of $b_{3,p} \cup b_{9,p} \in H^2(\Z[1/p], \Z_p(12))$ is non-zero at $p=691$ and is $0$ at $p \neq 691$. 

\end{enumerate}
\end{lem} 
\begin{proof}
The first statement is a consequence of the Quillen-Lichtenbaum theorem namely, $K_{22}(\Z)\otimes \Z_p  \cong H^2(\Z[1/p], \Z_p(12))$ and the fact that $K_{22}(\Z)=\Z/691$ (see \cite[Table 10.1.1]{[Weibel]}). The second statement follows from the related calculations in \cite{[McS]}.

\end{proof}

We are now ready to give a proof for Theorem \ref{theorem3}. Let $Y=\Ext^1(\Z(3), \Z(12))=\Ext^1(\Z, \Z(9))$. By the previous section we know that as a group 
\[ Y= \Z b_{9} \oplus \Z/ 2 .\]  
Here $b_{9}$ is the zeta element described in lemma $\ref{lem:Soule}$. By definition of $X$, we have a canonical projection map 
\[p: X \to Y .\] 
Therefore, as a set 
\[ X=  \bigsqcup_{a \in Y} X_{a} ,\]
where $X_{a}$ is the fiber of $p$ over a point $a$. In the following, our convention is as follows: for each notation appeared in the previous section we add a letter $a$ to indicate that the corresponding object depends on $a$. With this convention we have 
\[ \# \{x \in X| H_{\star, \diamond}(x) \leq B \} =\sum_{a \in Y} \# \{x \in X_{a}| H_{\star, \diamond} (x) \leq B \} .\] 
Suppose that $a$ correspond to $ub_{9}$ then the order of $\{ x \in X_a| H_{\star, \diamond}(x)<B \}$ is given by 
\[|T_a| \# \{m \in \Z| \exp(|r_{\infty}(ub_{9})|^{1/9}+(|m |\delta_{a}| r_{\infty}(b_{3})|)^{1/3}) \leq B \} .\]
We know that for all $a$ 
\[ \# T_{a}=\prod_{p} \# H^1(\Z[1/p], \Z_p(12))=\dfrac{\Sha(12)}{\zeta(-11)} .\] 
Therefore up to some power of $2$ we have
\[ \# \{x \in X| H_{\star, \diamond}(x) \leq B \} =\dfrac{\Sha(12)}{\zeta(-11)} \# \{(m,u) \in \N \times \N \}, \] 
where $(m,n)$ satisfies the following condition 
\[ \exp(|ur_{\infty}(b_{9})|^{1/9}+(|m |\delta_{a}| r_{\infty}(b_{3})|)^{1/3}) \leq B .\] 
In this particular case, $\delta_{a,p}$ can only be nontrivial when $p=691$ by Lemma $\ref{cup}$. For simplicity, let us define 
\begin{align*}
S(B) &= \{(m,u) \in \N \times \N | \exp(|ur_{\infty}(b_{9})|^{1/9}+(|m |\delta_{a}| r_{\infty}(b_{3})|)^{1/3}) \leq B \} \\
     &=\{(m,u) \in \N \times \N| |ur_{\infty}(b_{9})|^{1/9}+(|m |\delta_{a}| r_{\infty}(b_{3})|)^{1/3}) \leq \log(B) \}
\end{align*}
As $\delta_{a}$ depends on $a$ we consider two cases.

\textbf{Case 1: $691|u$.} Then, $\delta_{a,p}=0$ for all $p$. In this case 
\[ X_{a}=\Z b_{3} \oplus T_{a} .\]  

\textbf{Case 2: $691 \nmid u$.} Then $\delta_{a,p} \neq 0$ when $p=691$ by Lemma $\ref{cup}$. In this case $X_{a}$ is a little smaller then the previous case; i.e 
\[ X_{a}=691 \Z b_{3} \oplus T_{a} .\] 
Combining these two cases we have 
\[ |S(B)|=|S_1(B)|+|S_2(B)|-|S_1(B) \cap S_2(B)| ,\] 
where 
\[S_1=\{(m,u) \in \N \times \N| |ur_{\infty}(b_{9})|^{1/9}+(|m |\delta_{a}| r_{\infty}(b_{3})|)^{1/3}) \leq \log(B); 691|u \} \]
and 
\[S_2=\{(m,u) \in \N \times \N| |ur_{\infty}(b_{9})|^{1/9}+(|m |\delta_{a}| r_{\infty}(b_{3})|)^{1/3}) \leq \log(B); 691|m \}. \]
By letting $u=691u'$ in the first case we have 
\[S_1=\{(m,u') \in \N \times \N| |691u' r_{\infty}(b_{9})|^{1/9}+(|m |\delta_{a}| r_{\infty}(b_{3})|)^{1/3}) \leq \log(B) \}. \]
Therefore, by the counting lemma $\ref{lem: integer}$ we have 
\[ |S_1(B)| \sim \dfrac{1}{691 {12 \choose 3}} \dfrac{1}{r_{\infty}(b_9) r_{\infty}(b_3)} \log(B)^{12} .\] 
Similarly, we have 
\[ |S_2(B)| \sim  \dfrac{1}{691 {12 \choose 3}} \dfrac{1}{r_{\infty}(b_9) r_{\infty}(b_3)} \log(B)^{12} ,\]
and 
\[ |S_1(B) \cap S_2(B)|=\dfrac{1}{691^2 {12 \choose 3}} \dfrac{1}{r_{\infty}(b_9) r_{\infty}(b_3)} \log(B)^{12} .\] 
Consequently, we have
\begin{align*} 
|S(B)| & \sim \dfrac{1}{{12 \choose 3}} \dfrac{1}{r_{\infty}(b_9) r_{\infty}(b_3)} \left(\dfrac{2}{691}-\dfrac{1}{691^2} \right) \log(B)^{12} \\
       & \sim \dfrac{1}{8! 2! {12 \choose 3}} \dfrac{\Sha(3)}{\zeta(3)} \dfrac{\Sha(9)}{\zeta(9)} \left(\dfrac{2}{691}-\dfrac{1}{691^2} \right) \log(B)^{12}. 
\end{align*}
In conclusion 
\[ \# \{x \in X| H_{\star, \diamond} (x) \leq B \} \sim \dfrac{1}{ 8! 2! {12 \choose 3}} \dfrac{\Sha(3)}{\zeta(3)}  \dfrac{\Sha(9)}{\zeta(9)} \dfrac{\Sha(12)}{\zeta(-11)} \left(\dfrac{2}{691}-\dfrac{1}{691^2} \right) \log(B)^{12} .\]

\section*{Acknowledgements}

I am extremely grateful to Professor Kazuya Kato for introducing to me this topic of study, for sharing his insights, and for his constant support during the elaboration of this work. I am also thankful to the editor and the referee for some comments and suggestions which greatly improve the presentation of this article.


\end{document}